\documentclass[12pt]{amsart}
\usepackage{amssymb}
\usepackage{verbatim}
\usepackage[usenames]{color}
\usepackage{hyperref}
\usepackage{url}
\usepackage[all]{xy}

\newtheorem{theorem}{Theorem}[section]
\newtheorem*{theorem*}{Theorem}
\newtheorem{proposition}[theorem]{Proposition} 
 
\newtheorem{lemma}[theorem]{Lemma}

\newtheorem{fact}[theorem]{Fact}

\newtheorem{counterexample}[theorem]{Counterexample}

\theoremstyle{definition}
\newtheorem{definition}[theorem]{Definition}

\newtheorem{question}[theorem]{Question}

\theoremstyle{remark}

\newcommand{\defemp}{\textit}

\newcommand{\forces}{\Vdash}
\newcommand{\zf}{\textrm{ZF}}
\newcommand{\zfc}{\textrm{ZFC}}

\newcommand{\dom}{\textnormal{Dom}}

\newcommand{\mf}{\mathfrak}
\newcommand{\mc}{\mathcal}
\newcommand{\mbb}{\mathbb}

\newcommand{\p}{\mathcal{P}}

\title{Weak Distributivity Implying Distributivity}
\author{Dan Hathaway}

\address{Mathematics Department \\
University of Denver\\
Denver, CO 80208, U.S.A.}

\email{Daniel.Hathaway@du.edu}

\thanks{A portion of the results of this paper
 were proven during the September 2012 Fields Institute Workshop
 on Forcing while the author was supported by the Fields Institute.
Work was also done wile under NSF grant DMS-0943832.
}

\begin{document}

\begin{abstract}
Let $\mbb{B}$ be a complete Boolean algebra.
We show
 that if $\lambda$ is an infinite cardinal
 and $\mbb{B}$ is weakly $(\lambda^\omega, \omega)$-distributive,
 then $\mbb{B}$ is $(\lambda, 2)$-distributive.
Using a similar argument, we show that if $\kappa$
 is a weakly compact cardinal such that
 $\mbb{B}$ is weakly $(2^\kappa, \kappa)$-distributive
 and $\mbb{B}$ is $(\alpha, 2)$-distributive for each $\alpha < \kappa$,
 then $\mbb{B}$ is $(\kappa, 2)$-distributive.
\end{abstract}

\maketitle

\section{Introduction}
Given sets $A$ and $B$,
 ${^A B}$ denotes the set of functions
 from $A$ to $B$.
In this article,
 $\lambda$ and $\kappa$ will denote ordinals,
 although usually they can be assumed to be
 infinite cardinals.
As defined in \cite{Jech2},
 given $\lambda$ and $\kappa$,
 we say that a complete Boolean algebra $\mbb{B}$ is
 \emph{$(\lambda, \kappa)$-distributive} iff
 $$\prod_{\alpha < \lambda}
 \sum_{\beta < \kappa} u_{\alpha, \beta} =
 \sum_{f : \lambda \to \kappa}
 \prod_{\alpha < \lambda} u_{\alpha, f(\alpha)}$$
 for any $\langle u_{\alpha, \beta} \in \mathbb{B} :
 \alpha < \lambda, \beta < \kappa \rangle$.
Given maximal antichains
 $A_1, A_2 \subseteq \mbb{B}$, we say that
 $A_2$ \textit{refines} $A_1$ iff
 $(\forall a_2 \in A_2)(\exists a_1 \in A_1)\,
 a_2 \le_{\mbb{B}} a_1$.
It is a fact that
 $\mbb{B}$ is $(\lambda, \kappa)$-distributive
 iff each size $\lambda$ collection of maximal
 antichains in $\mbb{B}$ each of size $\kappa$ has
 a common refinement.
There is also a useful characterization in terms
 of forcing (which can be found in
 \cite{Jech2} as Theorem 15.38):
\begin{fact}
\label{distfact}
A complete Boolean algebra $\mbb{B}$ is
 $(\lambda, \kappa)$-distributive iff
 $$1 \forces_{\mbb{B}}
 (\forall f : \check{\lambda} \to \check{\kappa})\,
 f \in \check{V}.$$
\end{fact}

Unfortunately, the definition of
 weakly distributive varies in the
 literature
 (for example \cite{Kamburelis}).
We will use the one given by
 Jech (see \cite{Jech2}).
That is, we say that a
 complete Boolean algebra $\mbb{B}$
 is \emph{weakly $(\lambda,\kappa)$-distributive} iff
 $$\prod_{\alpha < \lambda}
 \sum_{\beta < \kappa}
 u_{\alpha, \beta} =
 \sum_{g : \lambda \to \kappa}
 \prod_{\alpha < \lambda}
 \sum_{\beta < g(\alpha)}
 u_{\alpha, \beta}.$$
This definition
 has a natural characterization
 in terms of forcing.
Given a set $X$ and
 $f, g : X \to \kappa$,
 we write $f \le g$ iff
 $g$ everywhere dominates $f$.
That is,
 $$(\forall x \in X)\,
 f(x) \le g(x).$$
\begin{fact}
\label{weakdistfact}
A complete Boolean algebra $\mbb{B}$
 is weakly $(\lambda,\kappa)$-distributive iff
 $$1 \forces_{\mbb{B}}
 (\forall f : \check{\lambda} \to \check{\kappa})
 (\exists g : \check{\lambda} \to \check{\kappa})\,
 g \in \check{V} \wedge f \le g.$$
\end{fact}
We will show the
 following:

\begin{theorem*}[A]
Let $\lambda$ be an infinite cardinal.
If
\begin{itemize}
\item[1)] $\mbb{B}$ is weakly $(\lambda^\omega, \omega)$-distributive,
\end{itemize}
 then $\mbb{B}$ is $(\lambda,2)$-distributive.
\end{theorem*}

\begin{theorem*}[B]
Let $\kappa$ be a weakly compact cardinal.
If
\begin{itemize}
\item[1)] $\mbb{B}$ is weakly $(2^\kappa, \kappa)$-distributive and
\item[2)] $\mbb{B}$ is $(\alpha, 2)$-distributive
 for each $\alpha < \kappa$,
\end{itemize}
 then $\mbb{B}$ is $(\kappa,2)$-distributive.
\end{theorem*}

We will then discuss why
 Theorem (B) does not hold when we have
 $\kappa = \omega_1$ instead of
 $\kappa$ being weakly compact,
 and we will show one way to fix the situation
 using the tower number.
Finally, we use the same
 idea using the
 tower number
 to prove a variation of Theorem (A)
 involving weak
 $(\lambda^\kappa,\kappa)$-distributivity
 for $\kappa > \omega$.

\section{Functions from $\lambda^\omega$ to $\omega$}

The proof of the following lemma
 uses the fact that well-foundedness of
 trees is absolute.
It is crucial, for what follows,
 that this lemma does not require
 ${^\omega \lambda} \subseteq M$.
See \cite{Hathaway1} for motivation
 and discussion.

\begin{lemma}
\label{lemma1}
For each $A \subseteq \lambda$,
 there is a function $f: {^\omega \lambda} \to \omega$
 such that whenever $M$ is a transitive model of
 $\zf$ such that $\lambda \in M$ and some
 $g : {(^\omega \lambda)^M} \to \omega$ in $M$ satisfies
 $$(\forall x \in {(^\omega \lambda)^M})\, f(x) \le g(x),$$
 then $A \in M$.
\end{lemma}
\begin{proof}
Fix $A \subseteq \lambda$.
Define $f: {^\omega \lambda} \to \omega$ by
 $$f(x) := \begin{cases}
 0 & \mbox{if } (\forall n < \omega)\, x(n) \not\in A,\\
 n+1 & \mbox{if } x(n) \in A \mbox{ but }
  (\forall m < n)\, x(m) \not\in A.
 \end{cases}$$
Let $M$ be a transitive model of $\zf$
 such that $\lambda \in M$
 but $A \not\in M$.
Suppose, towards a contradiction,
 that there is some $g \in M$
 satisfying $(\forall x \in (^\omega \lambda)^M)\,
 f(x) \le g(x)$.
Let $B$ be the set
 $$B := \{ t \in {^{<\omega} \lambda} :
 g(x) \ge |t| \mbox{ for all } x \mbox{ in $M$ extending } t \}.$$
Notice that $B \in M$.
Let $T \subseteq {^{<\omega} \lambda}$
 be the set of elements of $B$ all of whose
 initial segments are also in $B$.
Note that $T$ is a tree and $T \in M$.

For all $a \in \lambda$,
 $a \in A$ implies $\langle a \rangle \in B$.
Thus, there must be some $a_0 \in \lambda$
 such that $a_0 \not\in A$ but
 $\langle a_0 \rangle \in B$.
If there was not, then $A$ could be defined in $M$
 by $A = \{ a \in \lambda : \langle a \rangle \in B \}$.

Next, for all $a \in \lambda$,
 $a \in A$ implies $\langle a_0, a \rangle \in B$.
Thus, by similar reasoning as before,
 there must be some $a_1 \in \lambda$ such that
 $a_1 \not\in A$ but $\langle a_0, a_1 \rangle \in B$.
Continuing like this, we can construct a sequence
 $x \in {^\omega \lambda}$ satisfying
 $(\forall n < \omega)\, x \restriction n \in B$.
Thus,
 $(\forall n < \omega)\, x \restriction n \in T$,
 so $T$ is not well-founded.

Since well-foundedness is absolute,
 there is some path $x'$ through $T$ in $M$.
Since $(\forall n < \omega)\, x' \restriction n \in B$,
 we have $(\forall n < \omega)\, g(x') \ge n$,
  which is impossible.
\end{proof}

This implies the following lemma,
 whose order of quantifiers is not as powerful,
 but the functions have the ordinal $(\lambda^\omega)^M$
 instead of the set of sequences
 $(^{\omega} \lambda)^M$
 as their domains:

\begin{lemma}
\label{straitomegalemma}
Let $M$ be a transitive model of $\zf$
 such that the ordinal
 $\lambda$ is in $M$
 and $({^\omega \lambda})^M$ can be well-ordered in $M$.
Assume that for each $f : (\lambda^\omega)^M \to \omega$
 there is some
 $g : (\lambda^\omega)^M \to \omega$ in $M$
 such that $f \le g$.
Then $\p(\lambda) \subseteq M$.
\end{lemma}
\begin{proof}
Consider any $A \in \p(\lambda)$.
Use the lemma above with $A$ to get
 $\tilde{f} : {^\omega \lambda} \to \omega$
 such that if
 $\tilde{g} : ({^\omega \lambda})^M \to \omega$
 is any function in $M$ which satisfies
 \begin{equation}
 \label{eq1}
 (\forall x \in ({^\omega \lambda})^M)\,
 \tilde{f}(x) \le \tilde{g}(x),
 \end{equation}
 then $A \in M$.
Since $({^\omega \lambda})^M$ can be well-ordered in $M$,
 fix a bijection $$\eta :
  (\lambda^\omega)^M \to ({^\omega \lambda})^M$$ in $M$.
Define $f : (\lambda^\omega)^M \to \omega$ by
 $$f(\alpha) := \tilde{f}(\eta(\alpha)).$$
That is, the following diagram commutes:
 $$\xymatrix{
 ({^\omega \lambda})^M \ar[r]^{\tilde{f}} & \omega \\
 (\lambda^\omega)^M \ar[u]^{\eta} \ar[ur]_{f}.
 }$$

By hypothesis,
 let $g : (\lambda^\omega)^M \to \omega$
 be a function in $M$ which everywhere dominates $f$.
Define
 $\tilde{g} : ({^\omega \lambda})^M \to \omega$ by
 $$\tilde{g}(x) := g(\eta^{-1}(x)).$$
We have that $\tilde{g} \in M$ and
 $\tilde{g}$ satisfies \ref{eq1},
 so by the hypothesis on $\tilde{f}$,
 $A \in M$.
\end{proof}

We now have the main result of this section:
\begin{theorem*}[A]
Let $\mbb{B}$ be a complete Boolean algebra
 and $\lambda$ be an infinite cardinal.
If $\mbb{B}$ is weakly
 $(\lambda^\omega, \omega)$-distributive,
 then $\mbb{B}$ is $(\lambda,2)$-distributive.
\end{theorem*}
\begin{proof}
Let $\mu := \lambda^\omega$.
Assume $\mbb{B}$ is weakly $(\mu,\omega)$-distributive.
Force with $\mbb{B}$.
Every $f : \mu \to \omega$ in the extension
 can be everywhere dominated by some
 $g : \mu \to \omega$ in the ground model,
 so applying the lemma above
 in the extension
 (setting $M$ to be the ground model)
 tells us that
 the $\p(\lambda)$ of the extension is included
 in the ground model.
Hence,
 $\mbb{B}$ is $(\lambda,2)$-distributive.
\end{proof}

\section{Functions from $2^\kappa$ to $\kappa$ with $\kappa$ Weakly Compact}

The first lemma in the previous section was the key
 to the theorem there.
We have a parallel lemma here which,
 instead of using the absoluteness of trees being well-founded,
 uses the tree property to get similar absoluteness.
It is important that
 this lemma does not require
 ${^\kappa 2} \subseteq M$.
By weakly compact,
 we mean strongly inaccessible and
 having the tree property.

\begin{lemma}
\label{lemma2}
For each $a \in {^\kappa 2}$,
 there is a function
 $f: {^\kappa 2} \to \kappa$ such that
 whenever $M$ is a transitive model of $\zf$
 such that $\kappa \in M$,
 ${^{<\kappa}2} \subseteq M$,
 $({{^\kappa} 2})^M$ can be well-ordered in $M$,
 $(\kappa$ is weakly compact$)^M$,
 and some $g : ({^\kappa 2})^M \to \kappa$ in $M$
 satisfies
 $$(\forall x \in ({^\kappa 2})^M)\, f(x) \le g(x),$$
 then $a \in M$.
\end{lemma}
\begin{proof}
Fix $a \in {^\kappa 2}$.
Let $f : {^\kappa 2} \to \kappa$ be the function
 $$f(x) := \begin{cases}
 0 & \mbox{if } (\forall \alpha < \kappa)\, x(\alpha) = a(\alpha), \\
 \alpha+1 & \mbox{if } x(\alpha) \not= a(\alpha) \mbox{ but }
 (\forall \beta < \alpha)\, x(\beta) = a(\beta).
 \end{cases}$$
Let $M$ be an appropriate transitive model of $\zf$.
Suppose $g : ({^\kappa 2})^M \to \kappa$ in $M$ satisfies
 $(\forall x \in ({^\kappa 2})^M)\, f(x) \le g(x)$.
We will show that $a \in M$.

Suppose, towards a contradiction,
 that $a \not\in M$.
Let $$B := \{ t \in {^{<\kappa} 2} :
 g(x) \ge \dom(t) \mbox{ for all } x
 \mbox{ in $M$ extending t} \}.$$
Note that by definition,
 there cannot be any $x \in {^{\kappa}2}$ in $M$ satisfying
 $(\forall \alpha < \kappa)\, x \restriction \alpha \in B$
 because if there was such an $x$, we would have
 $(\forall \alpha < \kappa)\, g(x) \ge \alpha$,
 which is impossible.
Since $B$ need not be a tree,
 let $T \subseteq {^{< \kappa} 2}$ be the tree of
 those elements of $B$ all of whose initial segments
 are also in $B$.
Again, $T$ cannot have a length $\kappa$ path in $M$.
Note that for each $\alpha < \kappa$,
 $a \restriction \alpha \in B$.
This is because any
 $x \in {^{\kappa} 2}$ in $M$ which extends
 $a \restriction \alpha$ differs from $a$
 (since $a \not\in M$),
 and the smallest $\gamma$ such that
 $x(\gamma) \not= a(\gamma)$ must be $\ge \alpha$, so
 $$g(x) \ge f(x) = \gamma+1 > \gamma \ge \alpha = \dom(a \restriction \alpha).$$
Since
 $(\forall \alpha < \kappa)\, a \restriction \alpha \in B$,
 also
 $(\forall \alpha < \kappa)\, a \restriction \alpha \in T$.

Now, $B \in M$ (since ${^{<\kappa}2} \subseteq M$ and $g \in M$)
 and so $T \in M$.
Since $(\forall \alpha < \kappa)\, a \restriction \alpha \in T$,
 $(T$ has height $\kappa)^M$.
Since $(\kappa$ is strongly inaccessible$)^M$,
 we have $(T$ is a $\kappa$-tree$)^M$.
Since $(\kappa$ has the tree property$)^M$,
 there is a length $\kappa$ path through $T$ in $M$,
 which we said earlier was impossible.
\end{proof}

As before,
 this implies the following lemma,
 whose order of quantifiers is not as powerful,
 but the functions have the ordinal $({2^\kappa})^M$
 instead of the set of
 sequences $({{^\kappa}2})^M$
 as their domains:

\begin{lemma}
\label{weakcompstraitprop}
Let $M$ be a transitive model of $\zf$ such that
 the ordinal $\kappa$ is in $M$,
 ${^{<\kappa}2} \subseteq M$,
 $({^\kappa 2})^M$ can be well-ordered in $M$, and
 $(\kappa$ is weakly compact$)^M$.
Assume that for each
 $f : (2^\kappa)^M \to \kappa$
 there is some
 $g : (2^\kappa)^M \to \kappa$ in $M$
 such that $f \le g$.
Then $\p(\kappa) \subseteq M$.
\end{lemma}
\begin{proof}
The proof is similar to that
 of Lemma~\ref{straitomegalemma}.
\end{proof}

As before, the main result of this section follows:
\begin{theorem*}[B]
\label{weakcompact}
Let $\mbb{B}$ be a complete Boolean algebra and
 $\kappa$ be a weakly compact cardinal.
If $\mbb{B}$ is weakly $(2^\kappa, \kappa)$-distributive
 and $\mbb{B}$ is $(\alpha, 2)$-distributive for each
 $\alpha < \kappa$, then $\mbb{B}$ is $(\kappa,2)$-distributive.
\end{theorem*}
\begin{proof}
This follows from the lemma above
 just as Theorem (A) followed from Lemma~\ref{straitomegalemma}.
\end{proof}

\section{The Tower Number}

One might hope that Theorem (B)
 holds when $\kappa = \omega_1$ instead of
 $\kappa$ being weakly compact.
That is, one might hope that
 if a complete Boolean algebra $\mbb{B}$
 is weakly $(2^{\omega_1}, \omega_1)$-distributive
 and $(\omega,2)$-distributive,
 then it is $(\omega_1, 2)$-distributive.
Unfortunately,
 this cannot be proved in $\zfc$
 because $\mbb{B}$ could be a Suslin algebra
 (a Suslin algebra is c.c.c.\
 and therefore is weakly $(\lambda,\omega_1)$-distributive
 for any $\lambda$).
However,
 if we add the assumption that
 $1 \forces_\mbb{B} (\omega_1 < \mf{t})$,
 where we will define $\mf{t}$ soon,
 then $\mbb{B}$ \textit{is} $(\omega_1,2)$-distributive.
The argument is simpler than
 that of Theorem (A) and Theorem(B)
 and does not need the hypothesis
 of weak $(2^{\omega_1},\omega_1)$-distributivity.
As a final twist,
 we will combine several ideas to
 prove a variation of Theorem (A).

Recall that $\mathfrak{t}$,
 the \textit{tower number},
 is the smallest length of a sequence
 $$\langle S_\alpha \in [\omega]^\omega : \alpha < \kappa \rangle$$
 satisfying
 $(\forall \alpha < \beta < \kappa)\,
 S_\alpha \supseteq^* S_\beta$
 but there is no
 $S \in [\omega]^\omega$ satisfying
 $(\forall \alpha < \kappa)\,
 S_\alpha \supseteq^* S$
 (where $S_1 \subseteq^* S_2$ means
 $S_1 - S_2$ is finite).
It is not hard to see that
 $\omega_1 \le \mathfrak{t} \le 2^\omega$.
See \cite{Blass} for more on $\mathfrak{t}$ and
 related cardinals.
The following lemma is the key.
The idea is borrowed from Farah in \cite{Farah},
 who got the idea from Dordal in \cite{Dordal},
 who got the idea from Booth.

\begin{lemma}
\label{towerlemma}
Let $\kappa$ be such that
 $\omega_1 \le \kappa < \mf{t}$.
Let $M$ be a transitive model of $\zfc$
 such that $\kappa \in M$ and
 $(\forall \alpha < \kappa)\,
 \p(\alpha) \subseteq M$.
Then $\p(\kappa) \subseteq M$.
\end{lemma}
\begin{proof}
Fix $\kappa$ and $M$.
Since $\kappa \in M$ and
 $(\forall \alpha < \kappa)\, \p(\alpha) \subseteq M$,
 we have ${^{<\kappa}2} \subseteq M$.
Let $F : {^{<\kappa} 2} \to [\omega]^\omega$
 be a function in $M$ such that for all
 $t_1, t_2 \in {^{<\kappa} 2}$,
\begin{itemize}
\item[1)]
 $t_1 \sqsubseteq t_2 \Rightarrow
F(t_1) \supseteq^* F(t_2)$, and
\item[2)]
 $t_1 \perp t_2 \Rightarrow
 F(t_1) \cap F(t_2)$ is finite.
\end{itemize}
Such functions are easy to construct by induction
 (and the Axoim of Choice).
The construction will not get stuck at a limit
 stage $\gamma < \kappa$
 because given $t \in {^\gamma 2} \subseteq M$ and
 $\langle F(t \restriction \alpha) : \alpha < \gamma \rangle$,
 since $\gamma < \mf{t}$ there is some
 $S \in [\omega]^\omega \subseteq M$ such that
 $(\forall \alpha < \gamma)\,
 S \subseteq^* F(t \restriction \alpha)$.
The set $F(t)$ can be defined to be the least such $S$
 accoding to some fixed well-ordering of
 $[\omega]^\omega$.

Now, consider any $a \in {^\kappa 2}$.
We will show that $a \in M$.
The sequence
 $\langle F( a \restriction \alpha ) : \alpha < \kappa \rangle$
 is a $\supseteq^*$-chain (in V) of length $\kappa$.
Since $\kappa < \mf{t}$,
 fix some $S \in [\omega]^\omega$ satisfying
 $$(\forall \alpha < \kappa)\,
 S \subseteq^* F(a \restriction \alpha).$$
Since $\p(\omega) \subseteq M$,
 in particular $S \in M$.
Within $M$,
 the function $F$ and the set $S$ can be used together to define $a$:
 \[
 a = \bigcup \{ t \in {^{<\kappa}2} : S \subseteq^* F(t) \}. \qedhere
 \]
\end{proof}

By applying the lemma above inductively,
 we get an improvement:
\begin{lemma}
\label{pumpitcor}
Let $\kappa$ be such that
 $\omega_1 \le \kappa < \mf{t}$.
Let $M$ be a transitive model of $\zfc$
 such that $\p(\omega) \subseteq M$.
Then $\p(\kappa) \subseteq M$.
\end{lemma}

This last lemma is closely
 related to the fact that
 $2^\kappa = 2^\omega$
 whenever $\kappa < \mf{t}$.
A proof of this using an argument similar
 to Lemma~\ref{towerlemma}
 can be found in \cite{Blass}.
Martin's Axiom (MA) implies $\mf{t} = 2^\omega$,
 but the original proof \cite{Martin}
 that MA implies $2^\kappa = 2^\omega$
 whenever $\kappa < 2^\omega$
 used the almost disjoint coding poset.
We now have the application
 to complete Boolean algebras:
\begin{proposition}
\label{interesting_t_prop}
Let $\kappa$ be an infinite cardinal.
Let $\mbb{B}$ be a complete Boolean algebra
 such that $\mbb{B}$ is $(\omega,2)$-distributive
 and $1 \forces_\mbb{B} (\check{\kappa} < \mf{t})$.
Then $\mbb{B}$ is $(\kappa,2)$-distributive.
\end{proposition}
\begin{proof}
Apply Lemma~\ref{pumpitcor} in the forcing extension
 with $M$ equal to the ground model.
\end{proof}

Let $\kappa$ be such that $\omega_1 \le \kappa < \mf{t}$.
Any $A \in [\omega]^\omega$
 can be partitioned into $2^\omega$
 infinite sets with pairwise finite intersection.
Thus, fixing $\lambda \le 2^\omega$,
 the function $F : {^{<\kappa}2} \to [\omega]^\omega$
 in Lemma~\ref{towerlemma} can be replaced by a function
 $F : {^{<\kappa}\lambda} \to [\omega]^\omega$ satisfying
 the same conditions.
Slightly modifying the proof of Lemma~\ref{towerlemma},
 we get that if $M$ is a transitive model of $\zfc$
 such that $\lambda \in M$ and
 $(\forall \alpha < \kappa)\, {^\alpha \lambda} \subseteq M$,
 then ${^\kappa \lambda} \subseteq M$.
Inductively applying this fact
 yields an improvement:
\begin{lemma}
\label{againlemma}
Let $\kappa$ and $\lambda$ be such that
 $\omega_1 \le \kappa < \mf{t}$ and
 $\lambda \le 2^\omega$.
Let $M$ be a transitive model of $\zfc$ such that
 $\lambda \in M$ and
 ${^\omega \lambda} \subseteq M$.
Then ${^\kappa \lambda} \subseteq M$.
\end{lemma}

Now we may combine
 Lemma ~\ref{againlemma}
 with the argument in 
 Lemma~\ref{lemma1}.
The case $\kappa = \omega$
 of this next lemma is
 already handled by Lemma~\ref{lemma1}.
\begin{lemma}
\label{lemma3}
Let $\kappa$ and $\lambda$ be such that
 $\omega \le \kappa < \mf{t}$
 and $\lambda \le 2^\omega$.
For each $A \subseteq \lambda$,
 there is a function $f : {^{\kappa} \lambda} \to \kappa$
 such that whenever $M$ is a transitive model of $\zfc$ such that
 ${^{\omega}\lambda} \subseteq M$
 (and therefore ${^{\kappa} \lambda} \subseteq M$) and
 some $g : {^{\kappa} \lambda} \to \kappa$ in $M$
 satisfies $f \le g$,
 then $A \in M$.
\end{lemma}
\begin{proof}
Fix $\kappa$, $\lambda$, and $A$.
Define $f : {^{\kappa} \lambda} \to \kappa$ by
 $$f(x) := \begin{cases}
 0 & \mbox{if } (\forall \alpha < \kappa)\, x(\alpha) \not\in A, \\
 \alpha + 1 & \mbox{if } x(\alpha) \in A \mbox{ but }
 (\forall \beta < \alpha)\, x(\beta) \not\in A.
 \end{cases}$$
This is the analogue of the function $f$
 defined in Lemma~\ref{lemma1}.
Now fix $M$ and some $g : {^\kappa \lambda} \to \kappa$ in $M$
 satisfying $f \le g$.
Note that $2^\omega \in M$ so therefore $\kappa, \lambda \in M$.
Let $$B := \{ t \in {^{<\kappa}\lambda} : g(x) \ge \dom(t)
 \mbox{ for all } x \mbox{ extending } t\}.$$
Since ${^{<\kappa}\lambda} \cup
 \{ \kappa, \lambda, g \} \subseteq M$,
 also $B \in M$.

Assume towards a contradiction,
 that $A \not\in M$.
Arguing just as in Lemma~\ref{lemma1},
 there is some $x \in {^\kappa \lambda}$
 satisfying $(\forall \alpha < \kappa)\,
 x \restriction \alpha \in B$.
Since ${^\kappa \lambda} \subseteq M$,
 we have $x \in M$, and in particular
 $x$ is in the domain of $g$.
We now have $(\forall \alpha < \kappa)\,
 g(x) \ge \alpha$,
 which is impossible.
\end{proof}

\begin{lemma}
Let $\kappa$ and $\lambda$ be such that
 $\omega \le \kappa < \mf{t}$ and
 $\lambda \le 2^\omega$.
Let $M$ be a transitive model of $\zfc$
 such that ${^\omega \lambda} \subseteq M$
 (and therefore ${^\kappa \lambda} \subseteq M$).
Assume that for each $f : ({\lambda ^\kappa})^M \to \kappa$
 there is some
 $g : ({\lambda ^\kappa})^M \to \kappa$ in $M$
 satisfying $f \le g$.
Then $\p(\lambda) \subseteq M$.
\end{lemma}
\begin{proof}
This follows immediately
 from the previous lemma.
\end{proof}

Now follows the theorem:
\begin{theorem}
Let $\mbb{B}$ be a complete Boolean algebra.
Let $\kappa$ and $\lambda$ be such that
 $1 \forces_\mbb{B} ( \check{\kappa} < \mf{t} )$ and
 $1 \forces_\mbb{B} ( \check{\lambda} \le 2^\omega )$.
Assume that $\mbb{B}$ is $(\omega,\lambda)$-distributive
 and weakly $(\lambda^\kappa, \kappa)$-distributive.
Then $\mbb{B}$ is
 $(\lambda,2)$-distributive.
\end{theorem}

\section{Suslin Algebras and $\textrm{MA}(\omega_1)$}

The theorems in this paper relied on
 absoluteness results concerning trees.
We can get a counterexample to a generalization
 of these theorems by using a Suslin tree
 (a tree of height $\omega_1$ such that every
 branch and antichain is at most countable).
Recall the following definition:
\begin{definition}
A \defemp{Suslin algebra} is a complete Boolean algebra
 that is atomless, $(\omega,\kappa)$-distributive for
 each cardinal $\kappa$, and c.c.c.
\end{definition}
It is a theorem of $\zfc$ that
 there exists a Suslin algebra iff there exists
 a Suslin tree.
Furthermore, given a Suslin algebra $\mbb{B}$,
 there is a Suslin tree (turned upside down)
 that completely embeds into $\mbb{B}$,
 so $\mbb{B}$ is not $(\omega_1, 2)$-distributive
 (see \cite{Jech2}).
Immediately, we see that
 Theorem (B)
 cannot be changed
 by simply replacing the weakly compact cardinal
 $\kappa$ with $\omega_1$:
\begin{counterexample}
Let $\mbb{B}$ be a Suslin algebra.
Then $\mbb{B}$ is weakly $(2^{\omega_1}, \omega_1)$-distributive
 and $\mbb{B}$ is $(\omega,2)$-distributive,
 but $\mbb{B}$ is not $(\omega_1, 2)$-distributive.
\end{counterexample}
\begin{proof}
The only claim left to be verified is that
 $\mbb{B}$ is weakly $(2^{\omega_1}, \omega_1)$-distributive.
In fact, we will show that $\mbb{B}$ is
 weakly $(\lambda, \omega_1)$-distributive for all $\lambda$.
To see why, fix $\lambda$ and fix a $\mbb{B}$-name $\dot{f}$ such that
 $$1 \forces_{\mbb{B}} \dot{f} : \check{\lambda} \to \omega_1.$$
Since $\mbb{B}$ has the c.c.c., there are only countably many
 possible values for a given term in the forcing language.
In particular, for each $\alpha < \lambda$,
 there are only countably many possible values for
 $\dot{f}(\check{\alpha})$.
For each $\alpha < \lambda$, let
 $g(\alpha) < \omega_1$ be the supremum of
 these possible values.
We now have
 $$1 \forces_{\mbb{B}} (\forall \alpha < \check{\lambda})\,
 \dot{f}(\alpha) \le \check{g}(\alpha).$$
Since $\dot{f}$ was arbitrary,
 by Fact~\ref{weakdistfact}
 $\mbb{B}$ is weakly $(\lambda, \omega_1)$-distributive. 
\end{proof}

Unfortunately,
 the counterexample above used a Suslin algebra,
 which $\zfc$ does not prove exists.
In particular, we ask the following:
\begin{question}
Is it consistent with $\zfc$ that
 every complete Boolean algebra
 that is both $(\omega,\kappa)$-distributive
 for all $\kappa$ and
 weakly $(\lambda,\omega_1)$-distributive
 for all $\lambda$
 must also be $(\omega_1, 2)$-distributive?
\end{question}

The intuitive way to try to affirmatively answer the
 above question is to consider a model of $\textrm{MA}(\omega_1)$.
By Proposition~\ref{interesting_t_prop},
 we only need to worry about those $\mbb{B}$ such that
 $1 \forces_{\mbb{B}} (\omega_1 = \mf{t})$.
We present another result which shows we do not
 need to worry about complete Boolean algebras that
 satisfy both a strong chain condition and
 enough weak distributivity laws.
The main idea is the following:
 if we have a size $\lambda$ collection $\mc{C}$ of
 antichains in $\mbb{B}$ each of size $\kappa'$,
 then if $\mbb{B}$ is
 weakly $(\lambda, \kappa')$-distributive,
 then there is a maximal antichain $A \subseteq \mbb{B}$
 such that below each $a \in A$,
 each antichain in $\mc{C}$ has $< \kappa'$
 non-zero elements.
Assuming also that $\mbb{B}$ is $(\omega,|\mbb{B}|)$-distributive,
 we can repeatedly apply this construction
 countably many times until we produce a maximal antichain $B_{\omega}$
 such that below each $b' \in B_{\omega}$,
 each antichain of $\mbb{B}$ has only countably many non-zero elements.
That is, $B_{\omega}$ will witness that $\mbb{B}$ is
 ``locally c.c.c.''.
Then, we use a result of Baumgartner
 to conclude that
 since $\mbb{B}$ is locally c.c.c.\ and $(\omega, 2)$-distributive,
 $\mbb{B}$ is either $(\omega_1, 2)$-distributive
 or a Suslin tree can be embedded into $\mbb{B}$.
If we assume there are no Suslin trees
 (which follows from $\textrm{MA}(\omega_1)$),
 we get that $\mbb{B}$ must be $(\omega_1, 2)$-distributive.

\begin{theorem*}[D]
Assume there are no Suslin trees.
Let $\mbb{B}$ be a complete Boolean algebra such that
 $\mbb{B}$ is $(\omega, |\mbb{B}|)$-distributive,
 $\mbb{B}$ is $\kappa$-c.c.\ for some $\kappa < \aleph_{\omega_1}$,
 and $(\forall \textrm{uncountable } \kappa' < \kappa)\, \mbb{B}$ is
 weakly $(|\mbb{B}|^{\kappa'}, \kappa')$-distributive.
Then $\mbb{B}$ is $(\omega_1,2)$-distributive.
\end{theorem*}
\begin{proof}
We will construct a sequence of maximal antichains
 $$\langle B_n \subseteq \mbb{B} : n \in \omega \rangle$$
 such that $B_0 := \{ 1_{\mbb{B}} \}$ and
 $(\forall n < m < \omega)\, B_m$ refines $B_n$.
Each $B_n$ will have the property that for any maximal antichain $A$ below
 an element $b \in B_n$, for each $b' \in B_{n+1}$ extending $b$,
 $A$ will have $< |A|$ non-zero elements below $b'$.
We will then define the maximal antichain $B_\omega$ to refine each
 $B_n$, and we will argue that below each $b_{\omega} \in B_\omega$,
 $\mbb{B}$ is c.c.c.

Let $\kappa < \aleph_{\omega_1}$
 be the least cardinal such that
 $\mbb{B}$ is $\kappa$-c.c.
Define $B_0 := \{ 1_\mbb{B} \}$.
We will now define a maximal antichain
 $B_1 \subseteq \mbb{B}$ (which trivially refines $B_0$).
Every antichain in $\mbb{B}$ has size $< \kappa$.
Consider an uncountable cardinal
 $\kappa' = \aleph_\alpha < \kappa$.
Let $\lambda := |\mbb{B}|^{\kappa'}$.
Let $\langle A_\beta : \beta < \lambda \rangle$
 be an enumeration of the maximal antichains
 in $\mbb{B}$ of size $\kappa'$.
For each $\beta < \lambda$,
 let $\langle a_{\beta, \gamma} : \gamma < \kappa' \rangle$
 be an enumeration of the elements of $A_\beta$.
Let $\dot{G}$ be the canonical name for the generic filter.
Fix a name $\dot{f}$ such that
 $1 \forces \dot{f} : \check{\lambda} \to \check{\kappa'}$ and
 $$1 \forces (\forall \beta < \check{\lambda})\,
 \check{a}_{\beta, \dot{f}(\beta)} \in \dot{G}.$$
By hypothesis,
 $\mbb{B}$ is weakly $(\lambda, \kappa')$-distributive,
 so there is a maximal antichain $C_{0,\alpha} \subseteq \mbb{B}$
 (which trivially refines $B_0$) and a name $\dot{g}$ such that
 $1 \forces \dot{g} : \check{\lambda} \to \check{\kappa}'$ and
 $$1 \forces (\forall \beta < \check{\lambda})\,
 \dot{f}(\beta) \le \dot{g}(\beta).$$
Hence,
 $$1 \forces (\forall \beta < \check{\lambda})
 (\forall \gamma < \check{\kappa}')\,
 \gamma > \dot{g}(\beta) \Rightarrow
 \check{a}_{\beta, \gamma} \not\in \dot{G}.$$
This implies that
 below each $c \in C_{0,\alpha}$,
 each $A_\beta$ has
 $< |A_\beta| = \kappa'$ non-zero elements.
That is,
 for each $c \in C_{0, \alpha}$ and $A_\beta$,
 there are $< |A_\beta|$ many $a \in A_\beta$
 such that $c \wedge a \not= 0_{\mbb{B}}$.
 
For each $\aleph_\alpha < \kappa$,
 we have such a maximal antichain $C_{0,\alpha} \subseteq \mbb{B}$.
Since $\kappa < \aleph_{\omega_1}$,
 the family
 $\langle C_{0,\alpha} \subseteq \mbb{B} : \aleph_\alpha < \kappa \rangle$
 is countable.
Each $C_{0, \alpha}$ has size $\le |\mbb{B}|$,
 so since $\mbb{B}$ is $(\omega, |\mbb{B}|)$-distributive,
 we may fix a single maximal antichain $B_1 \subseteq \mbb{B}$
 which refines each $C_{0, \alpha}$.
Note that $B_1$ has the property that
 for each maximal antichain $A \subseteq \mbb{B}$ (below $1_\mbb{B}$)
 and $b' \in B_1$,
 $A$ has $< |A|$ non-zero elements below $b'$.

We will now define $B_2$.
Consider an uncountable cardinal $\kappa' = \aleph_\alpha < \kappa$.
Let $\lambda := |\mbb{B}|^{\kappa'}$.
Let $\langle A_{\beta} : \beta < \lambda \rangle$
 be an enumeration of all size $\kappa'$
 antichains that are each a partition of some element of $B_1$.
Since $\mbb{B}$ is weakly $(\lambda, \kappa')$-distributive,
 we may use a similar argument as before to get a
 maximal antichain $C_{1,\alpha}$ which refines $B_1$ such that
 below each $c \in C_{1,\alpha}$, each $A_\beta$ has
 $< |A_\beta| = \kappa'$ non-zero elements.
This completes the construction of $C_{1,\alpha}$.
As before, we may use the
 $(\omega, |\mbb{B}|)$-distributivity of $\mbb{B}$
 to get a common refinement $B_2$ of every
 maximal antichain in the family
 $\langle C_{1,\alpha} : \aleph_\alpha < \kappa \rangle$.
Note that $B_2$ has the property that
 for every partition $A$ of some element of $B_1$
 and $b' \in B_2$, $A$ has $<|A|$ non-zero elements below $b'$.

We may continue this procedure to get a sequence
 $\langle B_n : n \in \omega \rangle$ of maximal antichains
 of $\mbb{B}$.
The following diagram depicts the maximal antichains
 which we have constructed, where an arrow represents refinement:
$$
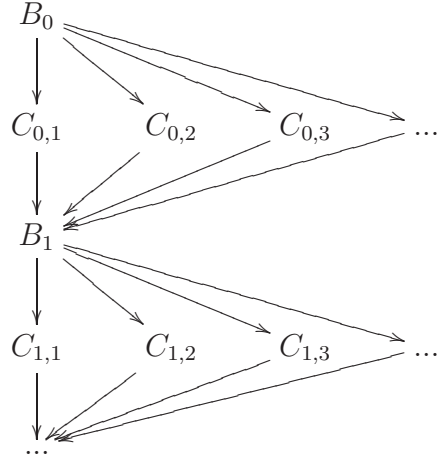
\xymatrix{
 B_0 \ar[d] \ar[dr] \ar[drr] \ar[drrr] \\
 C_{0,1} \ar[d] & C_{0,2} \ar[dl] & C_{0,3} \ar[dll] & {...} \ar[dlll] \\
 B_1 \ar[d] \ar[dr] \ar[drr] \ar[drrr] \\
 C_{1,1} \ar[d] & C_{1,2} \ar[dl] & C_{1,3} \ar[dll] & {...} \ar[dlll]  \\
 {...}
 }$$
Using the $(\omega, |\mbb{B}|)$-distributivity of $\mbb{B}$
 once more, we may get a single maximal antichain
 $B_\omega \subseteq \mbb{B}$
 which refines each $B_n$.
We will now argue that given any maximal antichain $A \subseteq \mbb{B}$
 and $b_\omega \in B_\omega$,
 $A$ has only countably many non-zero elements below $b$.

Fix an arbitrary maximal antichain $A_0 \subseteq \mbb{B}$.
Fix $b_{\omega} \in B_\omega$.
Let $\kappa_0 := |A_0|$.
If $\kappa_0 \le \omega$, we are done.
If not, let $b_1$ be the unique element of $B_1$ above $b_\omega$.
By the construction of $B_1$,
 $A_0$ has $< \kappa_0$ non-zero elements below $b_1$.
Let $\kappa_1 < \kappa_0$ be the number of such non-zero elements.
That is, letting
 $$A_1 := \{ a \wedge b_1 : a \in A_0 \},$$
 we have $|A_1| = \kappa_1 < \kappa_0$.
If $\kappa_1 \le \omega$, we are done because
 $|\{ a \wedge b_\omega : a \in A_0 \}| \le |A_1| \le \omega$.
Otherwise,
 let $b_2$ be the unique element of $B_2$ above $b_\omega$.
By the construction of $B_2$,
 $A_1$ has $< \kappa_1$ non-zero elements below $b_2$.
Let $\kappa_2 < \kappa_1$ be the number of such non-zero elements.
That is, letting
 $$A_2 := \{ a \wedge b_2 : a \in A_1 \},$$
 we have $|A_2| = \kappa_2 < \kappa_1$.
If $\kappa_2 \le \omega$, we are done by similar reasons as before.
If not, then we may continue the procedure.
However, the procedure will eventually terminate.
This is because if not,
 then we would have an infinite sequence of
 decreasing cardinals
 $$\kappa_0 > \kappa_1 > \kappa_2 > ...,$$
 which is impossible.
Thus, $A_0$ has only countably many non-zero elements below $b_\omega$.

At this point,
 we have argued that below the maximal antichain
 $B_\omega$, $\mbb{B}$ has the c.c.c.
Now, it must be that $\mbb{B}$ is $(\omega_1, 2)$-distributive.
Let us explain.
It suffices to show that $\mbb{B}$ is
 $(\omega_1,2)$-distributive below each element of $B_\omega$.
Fix any $b_\omega \in B_\omega$.
Below $b_\omega$, $\mbb{B}$ is c.c.c.\ and $(\omega,2)$-distributive.
Suppose, towards a contradiction,
 that $\mbb{B}$ is not $(\omega_1,2)$-distributive.
Quoting a result of Baumgartner
 \footnote{This was discovered independently by Andreas Blass
 who was told it was already proved by James Baumgartner.
 However, neither the author nor Blass have
 been able to find a proof in the literature.},
 there exists a Suslin tree which,
 when turned upside down, can be embedded into $\mbb{B}$
 below $b_\omega$.
However, we assumed there are no Suslin trees.
This completes the proof.
\end{proof}

\end{document}